\documentclass{article}
\usepackage{amsfonts,amsmath,amsthm,amssymb}
\usepackage{graphics, epsfig}
\usepackage{color}
\usepackage{appendix}
\usepackage{ulem}
\usepackage[makeroom]{cancel}
\usepackage{fancyhdr}
\usepackage{centernot}
\usepackage{tikz-cd}
\usepackage{mathtools}
\usepackage{ stmaryrd }

 \usepackage[usenames,dvipsnames]{pstricks}
 \usepackage{pst-grad} 
 \usepackage{pst-plot} 
\allowdisplaybreaks

\let\TeXchi\chi
\newbox\chibox
\setbox0 \hbox{\mathsurround0pt $\TeXchi$}
\setbox\chibox \hbox{\raise\dp0 \box 0 }
\def\chi{\copy\chibox}

\newtheorem{proposition}{Proposition}[section]
\newtheorem{theorem}{Theorem}[section]
\newtheorem{definition}{Definition}[section]

\newtheorem{lemma}{Lemma}[section]
\newtheorem{corollary}{Corollary}[section]
\newtheorem{remark}{Remark}[section]
\newtheorem{conjecture}{Conjecture}[section]

\newtheorem{problem}{Problem}[section]

\numberwithin{equation}{section}
\numberwithin{theorem}{section}
\numberwithin{definition}{section}
\numberwithin{example}{section}
\numberwithin{proposition}{section}
\numberwithin{lemma}{section}
\numberwithin{remark}{section}
\DeclareMathOperator{\Ima}{Im}

\DeclareMathOperator{\Id}{Id}
\DeclareMathOperator{\rk}{rank}

\DeclareMathOperator{\coker}{coker}
\DeclareMathOperator{\ind}{ind}
\newcommand\blfootnote[1]{%
  \begingroup
  \renewcommand\thefootnote{}\footnote{#1}%
  \addtocounter{footnote}{-1}%
  \endgroup
}
\begin{document}
\title{A conditional proof of the invariant subspace problem for quasinilpotent operators}
\author
{Manuel Norman}
\date{}
\maketitle
\begin{abstract}
\noindent The invariant subspace problem (ISP) is a well known unsolved problem in funtional analysis. While many partial results are known, the general case for complex, infinite dimensional separable Hilbert spaces is still open. It has been shown that the problem can be reduced to the case of operators which are norm limits of nilpotents. One of the most important subcases is the one of quasinilpotent operators, for which the problem has been extensively studied for many years. In this paper, we will introduce a new conjecture (supported by a heuristic argument), and we will prove conditionally that every quasinilpotent operator has a nontrivial invariant subspace. We will conclude by posing an open problem which would have deep implications regarding the ISP.
\end{abstract}
\blfootnote{Author: \textbf{Manuel Norman}; email: manuel.norman02@gmail.com
}
\section{Introduction}
The invariant subspace problem is one of the most important unsolved problems in functional analsysis. It asks whether every bounded linear operator $T \in B(X)$ ($X$ Banach space) has a nontrivial invariant subspace, i.e. a closed \footnote{Throughout this paper, we always tacitly assume that the subspaces we are talking about are closed.} linear subspace $W$ different from $X$ and $\lbrace 0 \rbrace$ such that $T(W) \subseteq W$. The finite dimensional case (with $\dim X \geq 2$) and the non separable (infinite dimensional) one are easy to settle. Enflo and Read were the first ones to construct counterexamples in the general setting. However, for reflexive Banach spaces (and in particular for Hilbert spaces) the problem is still open. We also note that the problem is open even in the real case, but here we will only deal with complex spaces.\\
A long-standing important subproblem of the ISP (which stands for invariant subspace problem) is the ISP for quasinilpotent operators, i.e. operators $T$ such that $\sigma(T)=\lbrace 0 \rbrace$. It is clear that solving the ISP for quasinilpotent operators also establishes the problem for every operator whose spectrum is a singleton: indeed, let $T$ be such that $\sigma(T)=\lbrace \lambda \rbrace$. Then, $T- \lambda \Id$ is quasinilpotent. Consequently, if we can show that quasinilpotent operators have some nontrivial invariant subspace $W$, we can say that, if $x \in W$, then $Tx- \lambda x = y \in W$. But then, $Tx=y + \lambda x \in W$. Thus, whenever $x \in W$, $Tx \in W$, so that $W$ is also a nontrivial invariant subspace for $T$. This subproblem has been subject of research in the last half century; we refer to the papers [30] and [7-8, 10, 18, 20] for some results.\\
In the next section we will recall some notions and some results that will be used in the proof of our main Theorem, which is shown in Section 4. The restrictions of the ISP obtained up to now and some consequences of our main Theorem are discussed in Section 3.\\
For more details on the ISP, we refer to [2-3, 26, 27]. Throughout the paper, $H$ will always denote a complex, infinite dimensional, separable Hilbert space.
\section{Preliminaries}
The results in the next sections involve some notions of (essential) spectra, which we recall here. First of all, the spectrum of an operator is defined as:
\begin{equation}\label{Eq:2.1}
\sigma(T):= \lbrace \lambda \in \mathbb{C}: T- \lambda \Id \text{is not invertible} \rbrace .
\end{equation} 
It is well know that this spectrum can be written as the disjoint union of three other spectra, namely the point spectrum $\sigma_p(T)$, which consists of the eigenvalues of $T$, the continuous spectrum $\sigma_c(T)$, consisting of the $\lambda$'s such that $T - \lambda \Id$ is injective, has dense range but is not surjective, and the residual spectrum $\sigma_r(T)$, which contains the $\lambda$'s such that $T-\lambda \Id$ is injective but does not have dense range. Another important kind of spectrum is the approximate point spectrum $\sigma_a(T)$, which contains the approximate eigenvalues of $T$. Approximate eigenvalues can also be characterised as those $\lambda$'s for which $T- \lambda \Id$ is not bounded below. It is clear that:
\begin{equation}\label{Eq:2.2}
\sigma_c(T)=\sigma_a(T) \setminus (\sigma_p(T) \cup \sigma_r(T)) .
\end{equation}
In order to introduce some essential spectra of $T$, we need the following definition:
\begin{definition}\label{Def:2.1}
An operator $T \in B(H)$ is Fredholm if its kernel and cokernel are finite-dimensional and its range is closed. This is equivalent to the fact that $T$ has a closed range and both $\dim \ker T$ and $ \dim \ker T^* $ are finite (see, for instance, [6, Exercise 8] or [23, Corollary 2.2]). The set of Fredholm operators is denoted by $\Phi$. The index of a Fredholm operator is defined by:
$$ \ind(T):=\dim \ker T - \dim \coker T = \dim \ker T - \dim \ker T^* . $$
The set of Fredholm operators with index $0$ is denoted by $\Phi_0$.
\end{definition}
The Fredholm spectrum $\sigma_{\Phi}(T)$ is defined as follows:
\begin{equation}\label{Eq:2.3}
\sigma_{\Phi}(T):= \lbrace \lambda \in \mathbb{C}: T- \lambda \Id \not \in \Phi \rbrace
\end{equation}
Moreover, the Weyl spectrum $\sigma_w(T)$ is:
\begin{equation}\label{Eq:2.4}
\sigma_{w}(T):= \lbrace \lambda \in \mathbb{C}: T- \lambda \Id \not \in \Phi_0 \rbrace
\end{equation}
The Weyl spectrum is invariant under compact perturbation, i.e. $\sigma_w(T)=\sigma_w(T+K)$ for every compact $K$. This easily follows from the following characterisation discovered by Schechter (see [29]):
\begin{equation}\label{Eq:2.5}
\sigma_{w}(T) = \bigcap_{K \in \mathcal{K}(H)} \sigma(T+K)
\end{equation}
where $\mathcal{K}(H)$ is the ideal of compact operators. A relation between the Fredholm spectrum, the Weyl spectrum and the spectrum of a bounded operator is given by the chain of inclusions:
$$ \sigma_{\Phi}(T) \subseteq \sigma_w(T) \subseteq \sigma(T) $$
Furthermore, it is proved in [6, Exercise 7] or [23, Remark 4.3] that the Fredholm spectrum is always nonempty when $H$ is, as in our case, infinite dimensional. When $T$ is quasinilpotent, this implies that $\sigma_{\Phi}(T) = \sigma_w(T) = \sigma(T) = \lbrace 0 \rbrace$. For more on these topics, we refer to [28].\\
Another important notion which is related to the spectrum is the one of pseudospectrum. For more information on this subject, we refer the reader to [11] for the case of square matrices and [31, Chapter 1, Section 4] for the case of linear operators acting on an infinite-dimensional space. We recall that, for an operator $T \in B(H)$ and for a real number $\epsilon >0$, the $\epsilon$-pseudospectrum is defined as:
\begin{equation}\label{Eq:2.6}
\sigma_{\epsilon}(T)= \lbrace z \in \mathbb{C} : \| R_{T}(z) \| > 1/ \epsilon \rbrace = \bigcup_{E: \, \, \| E \| < \epsilon} \sigma(T+E)
\end{equation}
where $R_T(z)$ is the resolvent operator defined on the resolvent set $\rho(T)= \mathbb{C} \setminus \sigma(T)$. We will now recall some important results which will be used in our proof. In [30], Tcaciuc proved that:
\begin{proposition}[30, Theorem 2.3]\label{Prop:2.1}
Let $T \in B(H)$ be quasinilpotent. Then, the following are equivalent:\\
(i) $T$ has a nontrivial invariant subspace;\\
(ii) there exists a rank $1$ operator $F$ such that $T + \alpha F$ is quasinilpotent for all $\alpha \in \mathbb{C}$;\\
(iii) there exists a rank $1$ operator $F$ such that $T + \alpha F$ is quasinilpotent for $\alpha = 1$ and for some other $\alpha \neq 0, 1$.
\end{proposition}
We will make use of this really useful characterisation together with the next proposition:
\begin{proposition}[30, Proposition 2.4]\label{Prop:2.2}
Let $T \in B(H)$ be quasinilpotent and $F$ be a rank $1$ operator. Then, exactly one of the following three possibilities happens:\\
(i) $T+ \alpha F$ is quasinilpotent for all $\alpha \in \mathbb{C}$;\\
(ii) for all nonzero $\alpha \in \mathbb{C}$, with possibly one exception, $\sigma_p(T + \alpha F)$ is countably infinite;\\
(iii) there is some natural number $K$ such that for all nonzero $\alpha \in \mathbb{C}$, $0 < |\sigma_p(T + \alpha F) \setminus \lbrace 0 \rbrace| < K$.
\end{proposition}
A fundamental ingredient of our proof is the following conjecture of Herrero's, which has been proved in [14-15] by Jiang and Ji.
\begin{proposition}[Herrero-Jiang-Ji Theorem]\label{Prop:2.3}
Let $T \in B(H)$ have a connected spectrum. Then, for every $\epsilon >0$, there exists a compact $K$ with $\| K \| < \epsilon$ such that $T=K+S$, where $S$ is strongly irreducible.
\end{proposition}
An operator is strongly irreducible if $ATA^{-1}$ is irreducible for every invertible operator $A$, otherwise it is strongly reducible. This is Definition 2.23 in [17]; we refer to this book for more details and other equivalent characterisations of these operators.\\
In our proof we will actually make use of a special case of the above result, which gives us some information about $S$. This is Lemma 2.10 in [14] applied to the case $\lambda = 0$.
\begin{proposition}[Herrero-Jiang-Ji Theorem for quasinilpotent operators]\label{Prop:2.4}
Let $T \in B(H)$ be quasinilpotent.  Then, for every $\epsilon >0$, there exist a compact $K$ with $\| K \| < \epsilon$ such that:\\
(i) $T=K+S$ for some strongly irreducible $S$;\\
(ii) $\dim \ker S =2$ and 
$$\ker S \subseteq \bigcap_{n \geq 1} \Ima (S^n)$$
(iii) $\overline{\operatorname{span}} \lbrace ker(S^n), n \geq 1 \rbrace = H$;\\
(iv) $\sigma_p(S^*)= \emptyset$.
\end{proposition}
It will be important in our proof to note that every nonempty set $A$ which is either finite (but not a singleton) or countably infinite is disconnected in $\mathbb{C} \cong \mathbb{R}^2$. This well known result can be intuitively seen to be true; we will not give a proof of this here.\\
The last result needed in our proof is the upper semicontinuity of the separated parts of the spectrum (the statement below follows from Theorem 3.16 in Chapter 4 of [19] and Theorem 2.14 in the same chapter):
\begin{proposition}\label{Prop:2.5}
Let $A \in B(H)$ be such that $\sigma(A)$ is disconnected. Let $\Gamma$ be a closed, simple, rectifiable curve separating $\sigma(A)$. Then, there exists $\delta >0$ such that, whenever $B \in B(H)$ is such that $\| B \| < \delta$, $\sigma(A+B)$ is still disconnected.
\end{proposition}
\begin{remark}\label{Rm:2.1}
\normalfont In [19], some expressions for $\delta$ are given. In particular, we notice that the following one can be used:
$$ \delta= \min_{\xi \in \Gamma} \frac{1}{\| R_{A} (\xi) \|}$$
Indeed, the proof of the upper semicontinuity of the spectrum given in [19] consists of three main steps \footnote{If not explicitely stated otherwise, all the results to which we refer in this Remark are given in Chapter 4 of [19].}: firstly, it is taken a value of $\delta>0$ such that the curve $\Gamma$ satisfies $\Gamma \subset P(B)$ (the resolvent set of $B$), and then Lemma 4.10 in Chapter I of [19] is applied (notice that the hypothesis that the projection depends \textit{continuously} on the parameter is satisfied by Theorem 3.15). From this we get Theorem 3.16, which gives us Proposition \ref{Prop:2.5} above via Theorem 2.14. Indeed, if $\| B \| < \delta$, then $\widehat{\delta}(A+B,A) \leq \| B \| < \delta$, so that Theorem 3.16 applies and leads to the desidered result. Now, $\Gamma \subset P(B)$ certainly holds under some assumptions. In the proof given in [19], the general case is considered: given two \textit{closed} operators $A$, $B$, and a curve $\Gamma$ as above, by the proof of Theorem 3.1 we may take
$$ \delta=\min_{\xi \in \Gamma} \frac{1}{2} \frac{1}{1+ |\xi|^2} \frac{1}{\sqrt{1 + \| R_{A} (\xi)\|^2}} $$
But in our case, we can be less general and get something easier to deal with. By Remark 3.2, if $B$ is a \textit{bounded} operator (as in our case), then we can take:
$$ \delta= \min_{\xi \in \Gamma} \frac{1}{\| R_{A} (\xi) \|}$$
\end{remark}
We conclude this section by stating the following conjecture, which is inspired by the results contained in [11] (in particular, by Theorem 6.1).
\begin{conjecture}\label{Conj:2.1}
Let $T$ be quasinilpotent. Then, there exist $a,b,t >0$ such that, for every rank $1$ quasinilpotent operator $F$ with $\| F \|=b$ for which $\widetilde{T}^* + \alpha F$ is not quasinilpotent for at least two distinct complex $\alpha$'s, there is some subset $S \subseteq \mathbb{C} \setminus \lbrace 0 \rbrace$ with $|S| \geq 2$ (which may depend on $F$) satisfying the following properties:\\
(i) $\widetilde{T}^* + \alpha F$ is not quasinilpotent for all $\alpha \in S$;\\
(ii) for all $\alpha \in S$, the following inclusion holds
$$ \sigma_{t}(\widetilde{T} + \alpha F) \subseteq \sigma(\widetilde{T} + \alpha F) + B(0, \Phi_{a,b,t,S}(\alpha)) $$
and $\sigma(\widetilde{T} + \alpha F) + B(0, \Phi_{a,b,t,S}(\alpha))$ is disconnected in $\mathbb{C}$.\\
Here, $\widetilde{T} = m T$, with $m >0$ such that $\| \widetilde{T} \| = a$, and $\Phi$ is a function depending on $\alpha$.
\end{conjecture}
Actually, it will be clear from the proof of our main result that even some slight modifications of this conjecture would imply the existence of a n.i.s. for every quasinilpotent operator.
\begin{remark}\label{Rm:2.2}
\normalfont We note that the first part of the conjecture is clearly true: the function $\Phi$, which is not required to be continuous or to satisfy any other condition, always exists. Moreover, it is easy to construct examples of quasinilpotent operators $A$ having some rank $1$ perturbations $A+ \alpha F$ for distinct values of $\alpha$ which are not quasinilpotent (use, e.g., Lemma 2.1 in [30]). So the actual question is whether $a,b,t>0$ exist so that for all rank $1$ quasinilpotents $F$ with $\| F \| =b$ for which $\widetilde{T} + \alpha F$ is not quasinilpotent for at least two nonzero distinct values of $\alpha$, $\sigma(\widetilde{T} + \alpha F) + B(0, \Phi_{a,b,t,S}(\alpha))$ is disconnected for all $\alpha$ in some set $S$ (not containing $0$) which may depend on $F$.\\
We also notice that we may instead restrict the Conjecture and ask for a set $S$ with only two elements satisfying the hypothesis, and this would be simpler, yet such statement might not hold. There is no assurance that it is always possible to take such a set $S$, while the heuristic argument we propose later (which justifies the Conjecture and let us expect it to be true) suggests that sets $S$ with $\geq 2$ elements and with the properties above probably exist.\\
To conclude, we point out that we cannot just take $\Phi$ "small enough" and prove the statement, because we do not have general knowledge of the spectrum of $\widetilde{T} + \alpha F$, nor of the $t$-pseudospectrum of this operator. Furthermore, these spectra vary with $F$, while $\Phi$ does not depend on $F$ but on its norm, on $t$, on the norm of $ \widetilde{T}$, on $S$ and on $\alpha$. Apart from the latter, the other values are fixed, and their existence is the focus of our conjecture. If $\Phi$ were allowed to depend on $F$, it would be much easier to show the statement, but as will be clear in the proof of our main result, we cannot do so. For these reasons, the problem we pose is actually more difficult that it may appear at first sight.
\end{remark}
\section{Some reductions}
In this section we will prove a useful reduction of the ISP for separable Hilbert spaces. We will first recall some related reductions obtained by other authors. It can be shown (see later) that if $\sigma_\Phi(T) \neq \sigma(T)$, then $T$ has a nontrivial hyperinvariant subspace (i.e. a nontrivial subspace which is invariant under every operator $S$ which commutes with $T$; clearly, such a space is also invariant under $T$). Moreover, it has been proved that every operator $T$ for which $\sigma_\Phi(T) = \sigma(T)$ is quasitriangular, i.e. there exists an increasing sequence $\lbrace P_n \rbrace$ of finite-rank projections converging strongly to $\Id$ such that
$$ \| P_n T P_n - T P_n \| \rightarrow 0 $$
The ISP is then equivalent to the ISP for biquasitriangular operators, i.e. operators $T$ which are quasitriangular and whose adjoint $T^*$ is quasitriangular. Actually, something more can be shown: we can consider just the operators with connected spectrum, since a disconnected spectrum is known to imply the existence of nontrivial (hyper)invariant subspaces (see [27, Corollary 2.11]). Moreover, we can assume that $0 \in \sigma(T)$ by translation by some scalar $\lambda$. Thus, we can define a class $C(H)$ consisting of the biquasitriangular operators with connected spectrum and connected Fredholm spectrum such that $0$ belongs to $\sigma(T)$. By Theorem 1.1 in [24], $C(H)= \mathcal{N}^{-}$, the norm closure of the space of nilpotent operators. Actually, the problem can be reduced a bit more, as we will see later. This result shows that understanding the norm closure of nilpotents leads to interesting consequences in the theory of invariant subspaces. There are various papers in which the norm closure of $\mathcal{N}$ has been studied: among these, we refer to [1, 9, 12, 24]. We refer to [5] for some results on invariant subspaces for quasitriangular operators.\\
We will prove our restriction of the ISP using the following extension of Lomonosov Theorem, obtained in [21]:
\begin{theorem}[22, Theorem A]\label{Thm:3.1}
Let $T \in B(H)$ be nonscalar. If there exists a nonzero compact $K$ such that $\rk(TK-KT) \leq 1$, then $T$ has a nontrivial hyperinvariant subspace.
\end{theorem}
In [22], Kim, Pearcy and Shields defined the class $\Delta(H)$ of operators for which there is some nonzero compact operator $K$ such that $\rk(TK-KT)=1$ and they proved some results. In particular, we recall the following ones:
\begin{proposition}[22, Proposition 1]\label{Prop:3.1}
$T \in \Delta(H)$ if and only if $(\alpha T + \beta \Id) \in \Delta(H)$ for all $\alpha \in \mathbb{C} \setminus \lbrace 0 \rbrace$ and for all $\beta \in \mathbb{C}$.
\end{proposition}
\begin{proposition}[22, Theorem 3]\label{Prop:3.2}
If $\sigma(T)$ is disconnected, then $T \in \Delta(H)$.
\end{proposition}
\begin{proposition}[22, Proposition 2]\label{Prop:3.3}
If $\sigma_p(T) \cup \sigma_p(T^*) \neq \emptyset$, then $T \in \Delta(H)$.
\end{proposition}
We can now prove:
\begin{theorem}\label{Thm:3.2}
Assuming Conjecture \ref{Conj:2.1}, the ISP for separable, complex, infinite dimensional Hilbert spaces is equivalent to the ISP for operators $T \in B(H)$ such that:
$$ \sigma(T)= \sigma_c(T)=\sigma_a(T)=\sigma_{\Phi}(T)=\sigma_w(T) $$
and for which the spectrum is connected, contains $0$ and is not a singleton.
\end{theorem}
\begin{proof}
Without loss of generality, we might consider $\sigma(T)$ to be connected and $\sigma_p(T) \cup \sigma_p(T^*) = \emptyset$. Otherwise, by Proposition \ref{Prop:3.2} and Proposition \ref{Prop:3.3}, it follows that $T \in \Delta(H)$ and hence $T$ has a nontrivial (hyper)invariant subspace as an application of Theorem \ref{Thm:3.1}.\\
Moreover, we can consider $\sigma_{\Phi}(T) = \sigma(T)$. Otherwise, let $\lambda$ be a complex number such that $T - \lambda \Id$ is not invertible but it is Fredholm. Since $\sigma_p(T) \cup \sigma_p(T^*) = \emptyset$, the operator $T - \lambda \Id$ is injective but not surjective, due to its non-invertibility. On the other hand, since $T- \lambda \Id$ is Fredholm and $T$ is nonscalar, the range of $T - \lambda \Id$ is a (closed) nontrivial invariant subspace for $T - \lambda \Id$, which implies the existence of a nontrivial invariant subspace for $T$.\\
Now, since $\sigma_{\Phi}(T) = \sigma(T)$, the set relationship $\sigma_{\Phi}(T) \subseteq \sigma_w(T) \subseteq \sigma(T)$ is actually the equality $\sigma_{\Phi}(T) = \sigma_w(T) = \sigma(T)$. Moreover, we can assume $\sigma_r(T) = \emptyset$, since otherwise the closure of the range of $T$ is a nontrivial invariant subspace. Therefore, we obtain:
$$ \sigma(T)=\sigma_p(T) \cup \sigma_c(T) \cup \sigma_r(T)= \sigma_c(T) $$
Now, by equation \eqref{Eq:2.2}, the continuous spectrum is equal to the approximate point spectrum, so that the chain of equalities in the Theorem can indeed be assumed. By Theorem \ref{Thm:4.1}, operators whose spectrum is a singleton can be excluded. To conclude, note again that $0$ can be assumed to be in the spectrum by translation by some scalar $\lambda$, and Theorem \ref{Thm:4.1} assures (under Conjecture \ref{Conj:2.1}) that we can exclude singletons.
\end{proof}
We also note that:
\begin{proposition}\label{Prop:3.4}
Every operator in the norm closure of $\mathcal{N}$ can be written as the norm limit of operators having nontrivial invariant subspaces.
\end{proposition}
\begin{remark}\label{Rm:3.1}
\normalfont This result can be easily shown just using the fact that nilpotent operators have nontrivial invariant subspaces. We have reported it here because this fact may be useful to contruct nontrivial invariant subspaces starting from the approximating sequence of nilpotents, as suggested in Question 1 at the end of [24].
\end{remark}
\section{Main result}
In this section we will prove the main result of the paper. To this end, we show the following relation between the spectrum, the point spectrum, and the Weyl spectrum of a bounded linear operator. This result can be found in [25, Chapter 8, equation (8.50)].
\begin{lemma}\label{Lm:4.1} Let $T \in B(H)$. Then:
$$ \sigma(T)= \sigma_p(T) \cup \sigma_w(T) $$
\end{lemma}
\begin{proof}
This is (8.50) in [25]. We will give here a short proof. First, note that:
$$\sigma_p(T) \cup \sigma_w(T) \subseteq \sigma(T)$$
Suppose that $\sigma_p(T) \, \cup \, \sigma_w(T) \subsetneq \sigma(T)$. Then, there is some $\lambda$ such that $T- \lambda \Id$ is Fredholm with index $0$ and is also injective. Since it is Fredholm, the closure of $\Ima(T- \lambda \Id)$ is equal to $\Ima(T- \lambda \Id)$ itself. By injectivity together with non-invertibility, $T- \lambda \Id$ is not surjective, so $\Ima(T- \lambda \Id) \neq H$ and hence the range is not dense in $H$. Thus, $T- \lambda \Id$ is injective but the range is not dense, which implies $\lambda \in \sigma_r(T)$ (the residual spectrum). It is well known that:
\begin{equation}\label{Eq:4.1}
\overline{\sigma_r(T)} \subseteq \sigma_p(T^*)
\end{equation}
where $\overline{A}:= \lbrace \overline{z} : z \in A \rbrace$ (the set of complex conjugates of the elements of $A$). Using this, we conclude that $\overline{\lambda} \in \sigma_p(T^*)$. But then $T^* - \overline{\lambda} \Id=(T- \lambda \Id)^*$ is not injective, so $\dim \ker ((T-\lambda \Id)^*) >0$. But this integer is equal to $\dim \ker (T- \lambda \Id)$ because this operator is Fredholm with index $0$. This would lead to a contradiction with the fact $\lambda \in \sigma_p(T)$. Therefore, the above equality holds.\\
A simple way to prove \eqref{Eq:4.1} is by using Hahn-Banach Theorem. More precisely, let $\lambda \in \sigma_r(T)$. By definition, the range of $T - \lambda \Id$ is not dense. By the Hahn-Banach Theorem, there is some nonzero $z \in X^*$ (the dual of $X$) that vanishes on $\Ima(T - \lambda \Id)$. $\forall x \in X$, we have:
$$ \langle z, (T- \lambda \Id)x \rangle = \langle (T^* - \overline{\lambda} \Id)z,x \rangle = 0 $$
Thus, $(T^* - \overline{\lambda} \Id)z = 0 \in X^*$ and hence $\overline{\lambda} \in \sigma_p(T^*)$. Note that this holds for any Banach space $X$, not only for Hilbert spaces. The above argument shows that:
$$ \overline{\sigma_r(T)} \subseteq \sigma_p(T^*) $$
which is \eqref{Eq:4.1}.
\end{proof}
\begin{remark}\label{Rm:4.1}
We recall that, for a Hilbert space $H$, the inclusion $\overline{\sigma_r(T)} \subseteq \sigma_p(T^*)$ follows from the orthogonal decomposition $H=\ker(T^* - \overline{\lambda} \Id) \, \oplus \,  \overline{\Ima(T - \lambda \Id)}$.
\end{remark}
At this point, assuming Conjecture \ref{Conj:2.1}, we provide a result on the existence of nontrivial invariant subspaces for the class of quasinilpotent operators.
\begin{theorem}\label{Thm:4.1}
Let $T \in B(H)$ be such that $\sigma(T)$ is a singleton. Then, assuming Conjecture \ref{Conj:2.1}, $T$ has a nontrivial invariant subspace (n.i.s.).
\end{theorem}
\begin{proof}
As we noted in the introduction, we only need to show this in the case when $T$ is quasinilpotent. We will prove that the unique possiblity that can occur in Proposition \ref{Prop:2.2} is the first one. This will imply, via Proposition \ref{Prop:2.1}, that $T$ has a nontrivial invariant subspace. For the time being, let $F$ denote any rank $1$ operator (later, we will choose a certain class of rank $1$ operators). Henceforth, we will write $T$ while actually referring to $m T$, where $m>0$ is such that $\| mT\| = a$ (given by Conjecture \ref{Conj:2.1}). Of course, for $m>0$, $mT$ having a n.i.s. is equivalent to $T$ having a n.i.s., so we can safely do this. Suppose that either (ii) or (iii) in Proposition \ref{Prop:2.2} happens (so, $T+ \alpha F$ is not quasinilpotent for all nonzero $\alpha$, with at most one exception, in view of Proposition \ref{Prop:2.1}). Now we prove that this assumption leads to a contradiction with (i) in Conjecture \ref{Conj:2.1}. As it is noticed at the beginning of the proof of Theorem 2.3 in [30], all the nonzero elements in $\sigma(T+ \alpha F)$ are eigenvalues. i.e. they belong to the point spectrum. This means that
$$ \sigma(T+ \alpha F)= \sigma_p(T+ \alpha F) \cup \lbrace 0 \rbrace $$
which is in fact a special case of Lemma \ref{Lm:4.1}.
Because of this result, our assumptions imply that $\sigma(T+\alpha F)$ is either finite (but not a singleton) or countably infinite. Thus, $\sigma(T+\alpha F)$ is not connected in $\mathbb{C}$. Since $T$ has a connected spectrum, by Herrero-Jiang-Ji Theorem we know that $T=K+S$ for some compact $K$ and strongly irreducible $S$. Here, we take:
$$ \epsilon = \frac{2}{t}$$
where $t>0$ is one of the values (or the unique value) of $t$ given by Conjecture \ref{Conj:2.1}.
Now that we have chosen $\epsilon$, we can apply Herrero-Jiang-Ji Theorem. The decomposition $T=K+S$ leads to $T^* = K^* + S^*$. Clearly, $T^*$ is still quasinilpotent and $K^*$ is still compact. We also recall that $\| A^* \| = \| A \|$ for any $A \in B(H)$. In order to conclude our proof, we will now choose a certain kind of rank $1$ operator $F$. To this end, we first show that $S$ is quasinilpotent. Indeed, since $\sigma_w(T)= \lbrace 0 \rbrace$, and since $K$ is compact, as an application of equation \eqref{Eq:2.5} we have $\sigma_w(S)= \lbrace 0 \rbrace$. Moreover, by Herrero-Jiang-Ji Theorem for quasinilpotents, $\sigma_p(S^*)= \emptyset$. Thus, since Schechter's characterisation of Weyl's spectrum implies $\sigma_w(S^*)= \lbrace 0 \rbrace$, by Lemma \ref{Lm:4.1} we have $\sigma(S^*)= \lbrace 0 \rbrace$, and hence $\sigma(S)= \lbrace0 \rbrace$ by the well known fact that $\sigma(A^*)=\overline{\sigma(A)}$. \footnote{If we use the Banach space adjoint, denoted here by $A'$, we have $\sigma(A')=\sigma(A)$. If instead we use, as in our case, the Hilbert space adjoint, the complex conjugation has to be added to the equality.}\\
Now take $F= e \otimes f$, with any nonzero $e \in \ker S = (\Ima S^*)^{\perp}$, $f \in \Ima S^*$. Then, consider:
$$ S^* x + \alpha \langle x, e \rangle f = \lambda x $$
Suppose $x \in \Ima S^*$. Then, the inner product is $0$ and we get $S^* x = \lambda x$, which is not possible. Suppose instead $x \not \in \Ima S^*$. Then, if the inner product is nonzero (otherwise, we can conclude as above), the LHS is in $\Ima S^*$ while the RHS is not, unless $\lambda = 0$. Therefore, $\sigma(S^* + \alpha F)= \lbrace 0 \rbrace$ for all $\alpha \in \mathbb{C}$.\\
Henceforth, $F$ will always denote a rank $1$ quasinilpotent operator as defined above, with norm $\| F \|=b$ (given by Conjecture \ref{Conj:2.1}). Whichever is $F$, by Conjecture \ref{Conj:2.1} we can find some $S$ with $|S \setminus \lbrace 0 \rbrace| \geq 2$ such that, for every nonzero $\alpha$ in $S$ (so, at least for two distinct nonzero complex $\alpha$'s), $\sigma(T^* + \alpha F) + B(0, \Phi_{a,b,c,t} (\alpha))$ is disconnected (and contains $\sigma_{t}(T^* + \alpha F)$). Thus, we can find a rectifiable, simple, closed curve $\Gamma$ such that the part of the complex plane inside it includes one of the disconnected region of this set, while all the other ones (or the other one) are outside it (in particular, for all $\xi \in \Gamma$, $\xi \not \in \sigma(T^* + \alpha F) + B(0, \Phi_{a,b,c,t} (\alpha)) \supseteq \sigma_{t}(T^* + \alpha F)$). As a consequence, $\Gamma$ separates the spectrum $\sigma(T^* + \alpha F)$ (because each connected component contains at least one point of $\sigma(T^* + \alpha F)$) and, since $\xi \not \in \sigma_{t}(T^* + \alpha F)$ as shown above:
$$ \| R_{T^* + \alpha F} \| \leq \frac{1}{t} $$
Since $\delta = \frac{1}{\| R_{T^* + \alpha F} \|}$, this implies:
$$ \epsilon= \frac{2}{t} < \frac{1}{\| R_{T^* + \alpha F} \|} = \delta $$
But then, $\| K^* \| = \| K \| < \epsilon < \delta$, and $\sigma(T^* + \alpha F)$ is separated by $\Gamma$ for $\alpha \in S \setminus \lbrace 0 \rbrace$, so that, by the upper semicontinuity of the separated parts of the spectrum, $\sigma(T^* + \alpha F - K^*)= \sigma(S^* + \alpha F)$ is disconnected, which is a contradiction, since $\sigma(S^* + \alpha F)=\lbrace 0 \rbrace$ is connected. Thus, $T^* + \alpha F$ is quasinilpotent for all rank $1$ quasinilpotents defined as above and with $\| F \| =b$, and hence the ISP for quasinilpotents holds true. We notice an important step in our proof: there is at most one exception to $\sigma(T^* + \alpha F)= \lbrace 0 \rbrace$ assuming that the ISP for quasinilpotents does not hold. Here, since $|S \setminus \lbrace 0 \rbrace | \geq 2$, this is not a problem: indeed, the conjecture holds for $T^* + \alpha F$ not quasinilpotent, so we have at most to exclude one value of $\alpha$ from our conclusion. But $S$ has at least two nonzero distinct elements, so that we have at least another nonzero value of $\alpha$ for which the conclusion holds. Thus, the proof of the ISP under Conjecture \ref{Conj:2.1} is concluded.
\end{proof}
\begin{corollary}\label{Crl:4.1}
Let $T$ be quasinilpotent. Then, assuming Conjecture \ref{Conj:2.1}, $\sigma(T+ \alpha F) = \lbrace 0 \rbrace$ for all complex $\alpha$'s, whenever $F=f \otimes e$ ($e \in \ker S$, $f \in \Ima (S^*)$, both nonzero) has norm $\| F \|=b$. Here, $S$ is given by Herrero-Jiang-Ji Theorem for quasinilpotents applied with $\epsilon= \frac{2}{t}$ ($t$ as in Conjecture \ref{Conj:2.1}).
\end{corollary}
\begin{proof}
The adjoint of $e \otimes f$ is $f \otimes e$, so this follows from the above proof.
\end{proof}
\section{Discussion of our Conjecture }
As already noticed in Remark \ref{Rm:2.2}, the actual conjecture is contained in the second part of Conjecture \ref{Conj:2.1}. We now give a heuristic argument which supports this conjecture. When $\| F \|$ is fixed, the spectrum $\sigma(T + \alpha F)$ consists of two parts (assuming it is not quasinilpotent): the first one being the point $0$, and the other one consisting of all the nonzero eigenvalues of $T + \alpha F$. Consider the following function $g: \mathbb{C} \setminus \lbrace 0 \rbrace \rightarrow \mathbb{C}$:
$$ g(z):= \langle R_{T}(z)f,e \rangle $$
This map, as noticed for instance in [30], is analytic (and hence holomorphic) on its domain, which is an open set. Moreover, by Lemma 2.1 in [30], $\lambda \in \mathbb{C} \setminus \lbrace 0 \rbrace$ is an eigenvalue of $T + \alpha F$ iff $g(\lambda)= \alpha^{-1}$. Consequently, for $\alpha$ such that $T + \alpha F$ is not quasinilpotent, $g$ is not constant. Thus, we can apply the Open Mapping Theorem, from which we deduce that we can consider $\alpha$ with $| \alpha |$ so large that the absolute value of some nonzero eigenvalue of $T + \alpha F$ is as large as we want. Indeed, say that we want some nonzero eigenvalue, say $\lambda$, such that $| \lambda | > R$ for some $R>0$ as large as we want. Then, take an open ball $B(x_0,r)$ with $|x_0|>0$ so large and $r>0$ so small that every point in $B(x_0,r)$ has absolute value $> R$. Then, by the Open Mapping Theorem, $g(B(x_0,r))$ is an open set, so that it is certainly $ \neq \lbrace 0 \rbrace$. Hence, there exists $\alpha$ with a large absolute value for which some nonzero eigenvalue of $T + \alpha F$ is as large as we want. Note that, for $R > \| T \|$, $\alpha$ will at least satisfy
$$ |\alpha| \geq \frac{R - \| T \|}{\| e \| \| f \|} $$
(this follows from Cauchy-Schwarz inequality applied to $g$ and from the fact that:
$$ \| R_{A}(z) \| < \frac{1}{|z| - \| A \|} $$
for $|z|> \| A \|$), but it might well be even larger.\\
Summing up, we can always find an eigenvalue with absolute value as large as we want, by choosing a suitably large (in absolute value) $\alpha$. We will now use the following result, which is a slight modification of a result in [4, pp. 130, Exercise 5]:
\begin{proposition}\label{Prop:5.1}
Let $f: D \subseteq \mathbb{C} \rightarrow \mathbb{C}$ be a function. Let $\rho, \phi >0$ and let $a \in \mathbb{C}$. Suppose that:\\
(i) $ 0 < 4 \phi <\rho/3$, $2 \phi < |a| < 3 \phi $;\\
(ii) $f(a) \neq 0$;\\
(iii) $f$ is analytic on $\overline{B(0, \rho)}$;\\
(iv) $|f(z)| \leq M$ for all $z \in \overline{B(0, \rho)}$.\\
Then, letting
$$ N:= |\lbrace z \in \overline{B(0, \rho/3)} \setminus \overline{B(0, 4 \phi)} : f(z)=0 \rbrace |$$
we have:
$$ N \leq \frac{\ln \left(\frac{M}{|f(a)|} \right)}{\ln \left( \frac{2}{1+\frac{|a|}{4 \phi}} \right)} $$
\end{proposition}
\begin{proof}
We first note that, as a consequence of the Identity Theorem, the number of zeros of $f$ in the compact set $\overline{B(0, \rho)} \setminus B(0, 4 \phi)$ is finite, so that $N < + \infty$. Consider the following map:
$$ h(z):= f(z) \prod_{k=1}^{N} \left( 1- \frac{z}{z_k} \right)^{-1} $$
where $z_k$ are all the zeros of $f$ in $\overline{B(0, \rho/3)} \setminus \overline{B(0, 4 \phi)}$. For $z$ such that $|z|=\rho$, we have:
$$ \left| \frac{z}{z_k} \right| \geq 3 $$
(because $|z_k| \leq \rho/3$), and hence:
$$ \left|1- \frac{z}{z_k} \right| \geq 2 $$
Therefore, we have:
$$ |h(z)| \leq M \cdot 2^{-N} $$
when $|z|=\rho$. Since $h$ is analytic (after removing the removable singularities) on the open, connected set $B(0,\rho)$, by the maximum modulus principle the above inequality holds on $\overline{B(0,\rho)}$. Now notice that:
$$ \left|1- \frac{a}{z_k} \right| \leq 1 + \frac{|a|}{4 \phi} $$
because $|z_k| \geq 4 \phi$ by hypothesis. Note that $1<1 + \frac{|a|}{4 \phi} <2$, since $|a| \in (2 \phi, 3 \phi)$. Thus, since $a$ is inside the ball, we have:
$$ |h(a)| \leq M \cdot 2^{-N} $$
which gives:
$$ |f(a)| \cdot \left(1 + \frac{|a|}{4 \phi} \right)^{-N} \leq |h(a)| \leq M \cdot 2^{-N} $$
that is, since $1<1 + \frac{|a|}{4 \phi} <2$ as noticed above:
$$ N \leq \frac{\ln \left(\frac{M}{|f(a)|} \right)}{\ln \left( \frac{2}{1+\frac{|a|}{4 \phi}} \right)} $$
as desidered.
\end{proof}
Now, continuing with the heuristic argument which supports Conjecture \ref{Conj:2.1}, choose some $R > \| T \|$ and find some eigenvalue $\lambda$ of $T^* + \alpha F$ for some large $\alpha$, as shown above. Fix this $\lambda$, let $\phi = 4 R $ and consider:
$$ h(z):= g(z+ \lambda) - g(\lambda)=g(z+\lambda)- \alpha^{-1} $$
This function is clearly analytic on $\mathbb{C} \setminus \lbrace - \lambda \rbrace$. Taking $\rho + \left( \frac{3}{2} + \frac{1}{1000} \right)\phi > |\lambda| > \rho +  \frac{3}{2} \phi $ \footnote{By a suitable choice of $x_0$ and $r$ as previously shown, we can find an eigevalue satisfying this inequality.}, the number of zeros of $h$ in $\overline{B(0, \rho/3)} \setminus \overline{B(0, 4 \phi)}$ is bounded above as in Proposition \ref{Prop:5.1}. Indeed, the ball centred in $\lambda$ with radius $\rho$ does not contain $0$ because $|\lambda| > \rho + \left( \frac{3}{2} \right)\phi > \rho$. Now, by choosing in a suitable way $x_0$ and $r$, we can also assume that $\Re(\lambda), \Im(\lambda) >0$, where $\Re$ and $\Im$ are the real part and the imaginary part, respectively. Fix some $a$ with $\Re(a), \Im(a) >0$ and with $|a| \in ( 2 \phi, 3 \phi)$. Moreover, we can always find $a$ satisfying this for which we also have $f(a) \neq 0$ (the number of zeros is at most countable since $f$ is analytic, so that one such $a$ certainly exists). Since $\rho + \left( \frac{3}{2} + \frac{1}{1000} \right)\phi > |\lambda|$, $a$ is contained in the ball $B(\lambda, \rho)$, that is, $a=z+\lambda$ for some $z \in B(0,\rho)$. Furthermore, since (for $z \in B(0,\rho)$) $|\lambda + z| \geq |\lambda| - |z| > \frac{3}{2} \phi$, we deduce that:
$$ |g(z + \lambda)| \leq \frac{\| e \| \| f \|}{|\lambda + z| - \| T \|} \leq \frac{\| e \| \| f \|}{\frac{3}{2} \phi - \| T \|} $$
Moreover, for $\rho$ large enough, $|\lambda|$ is so large that $|\alpha| \geq 1$, so letting
$$ M:= \frac{\| e \| \| f \|}{\frac{3}{2} \phi - \| T \|} + 1$$
by Proposition \ref{Prop:5.1} we can conclude that:
$$ N \leq \frac{\ln \left(\frac{M}{|f(a)|} \right)}{\ln \left( \frac{2}{1+\frac{|a|}{4 \phi}} \right)} $$
Now, if $\lambda_j$ denotes the $j$-th zero of $h$ in $\overline{B(0, \rho/3)} \setminus \overline{B(0, \rho/4)}$, let us consider the circle $C_j=\lbrace z \in \mathbb{C} : |z - \lambda|= r_j \rbrace$, where $r_j >0$ is the unique number such that $\lambda_j$ belongs to $C_j$, i.e. $r_j=|\lambda_j - \lambda|$. It is clear that the distance between the circles satisfies the following property: there exist $i, j$ distinct for which there is no $k \neq i,j$ such that $r_k$ is \footnote{WLOG, we suppose that $r_i < r_j$. If all the $r_j$'s have the same value or there is only one eigenvalue, skip this step and proceed with the last part of this heuristic argument.} in $( r_i,r_j)$ and such that:
$$ r_j - r_i \geq \frac{\rho/3 - 4 \phi}{N} \geq \frac{\rho/3 - 4 \phi}{\ln \left(\frac{M}{|f(a)|} \right)} \cdot \ln \left( \frac{2}{1+\frac{|a|}{4 \phi}} \right) $$
Since $\phi$, $|a|$ and $M$ are fixed and do not depend on $\rho$, by taking $| \lambda|$ and $\rho$ as large as needed, the RHS of the above inequality will get as large as we want (note that $a$ is always inside the ball $B(\lambda, \rho)$). Thus, the distance between two consecutive circles $C_i$, $C_j$ is as large as we want. Consequently, we can always find $\alpha$ so large in absolute value that there is an annulus for which the distance between its two circles $C_i$ and $C_j$ is as large as we want (if there is only one eigenvalue or all the $r_j$'s have the same value, an analoguous conclusion clearly holds). Now, the resolvent norm \textit{usually} has lower values \footnote{Of course, there is a lower bound on the resolvent norm, namely $\frac{1}{d(\xi, \sigma(T^* + \alpha F))}$ ($\xi \in \Gamma$), but notice that this gets smaller and smaller by taking a suitably large annulus, whichever is the rank $1$ quasinilpotent $F$ with norm equal to $b$, so this is not a problem.} when it is not near some point in the spectrum, and by taking $|\alpha|$ as large as needed we have an annulus as big as we want where there is no point of the spectrum of $T^* + \alpha F$. Hence, the argument above \textit{suggests} that there are some $\alpha$'s for which a bound on the resolvent norm \textit{for all rank $1$ quasinilpotents $F$ with norm $b$} should exist, because we can always find a region as big as we need without points of the spectrum (whichever is $F$ rank $1$ quasinilpotent with norm $b$), where we can thus find a closed, simple, rectifiable curve $\Gamma$ for which the resolvent norm at each of its points is always bounded by some constant (and such that $\Gamma$ divides the spectrum, since $\lambda$ is inside it, while $0$ is not surrounded by the curve). This is why we expect our Conjecture \ref{Conj:2.1} to hold.
\section{Further results and open questions}
Our Conjecture suggests the following, more general definition:
\begin{definition}\label{Def:6.1}
Let $A \in B(H)$ be a bounded operator on a complex, separable, infinite dimensional Hilbert space, and suppose that it has a connected spectrum. We say that $A$ is R-boundable (R stays for 'resolvent') if the following conditions are satisfied:\\
(i) there exists $a>0$ such that at least one rank $1$ quasinilpotent $F$ such that $\sigma(\widetilde{A} + \alpha F)$ is disconnected for all $\alpha \in D$ for some subset $D \subseteq \mathbb{C}$ with $|D| \geq 2$;\\
(ii) there exist $b,t >0$ such that, for every $F$ rank $1$ quasinilpotent as in (i), there is some subset $S \subseteq D$ with $|S| \geq 2$ (which may depend on $F$) such that 
$$\sigma_{t}(\widetilde{A} + \alpha F) \subseteq \sigma(\widetilde{A} + \alpha F) + B(0, \Phi_{a,b,t,S}(\alpha))$$
and $\sigma(\widetilde{A} + \alpha F) + B(0, \Phi_{a,b,t,S}(\alpha))$ is disconnected in $\mathbb{C}$.\\
Here, $\Phi_{a,b,t,S}$ is a function depending on $\alpha$ and $\widetilde{A}= m A$, with $m >0$ such that $\| \widetilde{A} \| = a$, and $\| F \| =b$.
\end{definition}
We first prove the following result, which justifies the name R-boundable:
\begin{proposition}\label{Prop:6.1}
Let $A$ be R-boundable. Then, there exists a closed, simple, rectifiable curve $\Gamma$ such that, for all $\xi \in \Gamma$, for every rank $1$ quasinilpotent $F$ as in (i), and for all $\alpha \in S$ (with $S$ which may change with $F$):
$$ \| R_{A+ \alpha F}(\xi) \| \leq \frac{1}{t} $$
where $t>0$ is any of the value given by Definition \ref{Def:6.1}.
\end{proposition}
\begin{proof}
By definition of R-boundable, for every rank $1$ quasinilpotent $F$ as in (i), and for all $\alpha \in S$ (with $S$ which may change with $F$), we have:
$$\sigma_{t}(\widetilde{A} + \alpha F) \subseteq \sigma(\widetilde{A} + \alpha F) + B(0, \Phi_{a,b,t,S}(\alpha))$$
with $\sigma(\widetilde{A} + \alpha F) + B(0, \Phi_{a,b,t,S}(\alpha))$ disconnected. Thus, we can find a simple, closed, rectifiable curve $\Gamma$ surrounding one of the connected components of this set. For all $\xi \in \Gamma$, we have (by the inclusion above):
$$ \| R_{A+ \alpha F}(\xi) \| \leq \frac{1}{t} $$
with $t>0$ as in Definition \ref{Def:6.1}.
\end{proof}
We also define a class of bounded operators on a complex, separable, infinite dimensional Hilbert space $H$, which we call $I(H)$, consisting of all the R-boundable operators $A \in B(H)$ for which the following holds:\\
$A$ has a n.i.s. $\Leftrightarrow$ there exists a rank $1$ operator $F$ such that $\sigma(A+ \alpha F)$ is connected for all $\alpha \in E$, where $E \subseteq S$ is some subset of $S$ ($S$ given by the definition of R-boundable) such that $|E \setminus \lbrace 0 \rbrace| \geq 2$\\
We notice that quasinilpotents belong to $I(H)$, assuming Conjecture \ref{Conj:2.1}.\\
We now pose the following problem:
\begin{problem}\label{Prob:6.1}
Does every biquasitriangular non-invertible operator in $B(H)$ with all the essential spectra equal to the spectrum belong to $I(H)$?
\end{problem}
The class of operators for which the ISP has not been established yet is contained in the class of biquasitriangular operators. We can, as usual, assume that $0$ belongs to their spectrum, by translation. Up to now, we have no argument in favour or against the above problem (except for the case of quasinilpotents). Finally, we provide a result on the existence of invariant subspaces for biquasitriangular operators belonging to the class $I(H)$.
\begin{theorem}\label{Thm:6.1}
Every biquasitriangular non-invertible operator $A \in I(H)$ for which all the essential spectra are equal to the spectrum has a n.i.s.
\end{theorem}
\begin{proof}
The proof is similar to the one of Theorem \ref{Thm:4.1}. The unique difference is that the special case of Herrero-Jiang-Ji Theorem does not apply to this general case, so we replace it with a combination of two other results: first apply Lemma 3.4 with $m=1$ in [14], and then use Lemma 2.3 in [13] together with the fact that strongly irreducible operators have connected spectrum. Everything else can be easily modified in order to conclude the proof of this result.
\end{proof}
\section{Conclusion}
In this paper we have established conditionally the invariant subspace problem for quasinilpotent operators. Moreover, we have also explained with a heuristic argument why we expect Conjecture \ref{Conj:2.1} to be true. We think that this paper will stimulate further research, in particular in the direction of Problem \ref{Prob:6.1} which, if solved affirmatively, would give a proof of the entire invariant subspace problem for complex Hilbert spaces via Theorem \ref{Thm:6.1}.

\section*{Acknowledgements}
I thank Professors Jonathan Partington, Daniel Rodr\'iguez Luis, March Boedihardjo, Vasile Lauric, Adi Tcaciuc and Aharon Atzmon for their useful comments on the previous versions of this paper, which helped me improve the article.

\end{document}